\documentclass[12pt]{amsart}
\usepackage{amssymb,amsmath,amsfonts,latexsym}
\usepackage{bm}
\usepackage[all,cmtip]{xy}
\usepackage{amscd}
\usepackage{setspace}
\usepackage{mathrsfs}
\usepackage[colorlinks ,citecolor=blue,urlcolor=blue,linkcolor=blue]{hyperref}

\allowdisplaybreaks

\newcommand{\bea}{\begin{eqnarray}}
	\newcommand{\eea}{\end{eqnarray}}

\newcommand{\clb}{\mathcal{B}}

\newcommand{\cle}{\mathcal{E}}

\newcommand{\clh}{\mathcal{H}}

\newcommand{\cln}{\mathcal{N}}

\newcommand{\clr}{\mathcal{R}}

\def\textmatrix#1&#2\\#3&#4\\{\bigl({#1 \atop #3}\ {#2 \atop #4}\bigr)}
\def\dispmatrix#1&#2\\#3&#4\\{\left({#1 \atop #3}\ {#2 \atop #4}\right)}
\newcommand{\be}{\begin{equation}}
	\newcommand{\ee}{\end{equation}}
\newcommand{\ben}{\begin{eqnarray*}}
	\newcommand{\een}{\end{eqnarray*}}

\newcommand{\NI}{\noindent}

\newcommand{\bi}{\begin{itemize}}
	\newcommand{\ei}{\end{itemize}}

\newcommand{\C}{\mathbb{C}}

\newcommand{\D}{\mathbb{D}}

\newcommand{\T}{\mathbb{T}}

\newcommand{\h}{{H}^2}

\theoremstyle{definition}

\theoremstyle{plain}

\newtheorem{thm}{Theorem}[section]
\newtheorem{cor}[thm]{Corollary}
\newtheorem{lem}[thm]{Lemma}
\newtheorem{prop}[thm]{Proposition}
\theoremstyle{definition}

\newtheorem{rem}[thm]{Remark}

\numberwithin{equation}{section}

\let\phi=\varphi

\begin{document}
\setcounter{page}{1}
	
\title[Partially isometric TTO and DTTO]{Partially isometric truncated and dual truncated Toeplitz operators}

\author[Babbar]{Kritika Babbar}
\address{Indian Institute of Technology Roorkee, Department of Mathematics,
		Roorkee-247 667, Uttarakhand,  India}
\email{kritika@ma.iitr.ac.in, kritikababbariitr@gmail.com}
	
\author[Javed]{Mo Javed}
\address{Indian Institute of Technology Roorkee, Department of Mathematics,
		Roorkee-247 667, Uttarakhand,  India}
\email{mo\_j@ma.iitr.ac.in, javediitr07@gmail.com}

\author[Maji]{Amit Maji}
\address{Indian Institute of Technology Roorkee, Department of Mathematics,
		Roorkee-247 667, Uttarakhand,  India}
\email{amit.maji@ma.iitr.ac.in, amit.iitm07@gmail.com ({Corresponding author)}}



\begin{abstract}
Let $\theta$ be a non-constant inner function and let $\phi=\overline{u}v$, where $u$ and $v$ are inner functions such that $v$ divides $\theta$. In this paper we characterize the partially isometric truncated Toeplitz operators $A_{\phi}$ and dual truncated Toeplitz operators $D_{\phi}$ with symbols of the form $\phi=\overline{u}v$. Along with that, we obtain a few more characterization results, including the space of extremal vectors for non-zero partially isometric truncated and dual truncated Toeplitz operators.
\end{abstract}
	
\maketitle

\section{Introduction}

The study of truncated Toeplitz operators (TTOs in short) has its origins in the seminal paper of Sarason (\cite{DS}, 2007), who introduced a systematic framework for these operators and developed many of their fundamental properties. Given an inner function $\theta$, the corresponding model space $K_{\theta}=H^2\ominus\theta H^2$ serves as the natural domain for these operators, where $H^2$ denotes the Hardy space over the unit disk $\D$, and they are defined as the compression of the Laurent operators acting on the $L^2(\T)$ to the model space $K_{\theta}$. 
Sarason's work stimulated considerable interest in the operator-theoretic community, leading to rapid developments in comprehending the algebraic and spectral properties of TTOs. A central theme in Sarason's work is characterizing when TTOs belong to specific operator classes. Cima and Garcia \cite{CIMA-GARCIA-ROSS-WOGEN} studied the structure of these operators in connection with complex symmetric operators, demonstrating that every TTO is complex symmetric with respect to a canonical conjugation on $K_{\theta}$. They also investigated their properties on the finite dimensional model spaces \cite{CIMA-ROSS-WOGEN}. Later, several authors investigated these operators. We mention a few relevant works: \cite{B-C-F-M-T,BARANOV-BESSONOV-KAPUSTIN,G-ROSS-W,S-T-Z} including the recent survey \cite{ICF} and the references therein. Sedlock \cite{NS} subsequently established a complete characterization of idempotent TTOs while investigating the problem of when the product of two truncated Toeplitz operators is a truncated Toeplitz operator. More precisely, he introduced the so-called Sedlock classes which form maximal abelian subalgebras within the collection of bounded TTOs and provide a powerful framework for studying algebraic properties of TTOs.	
Alongside TTOs, the dual setting corresponding to orthogonal complements of model spaces gave rise to dual truncated Toeplitz operators (DTTOs in short). Ding and Sang \cite{XD} introduced these operators and established their basic structural properties along with the zero product problem of two DTTOs. These operators are the compression of the Laurent operators acting on the $L^2(\T)$ to the orthogonal complement of the model space, that is, $K_{\theta}^{\perp}=L^2(\T)\ominus K_{\theta}$.   
Later development includes the Brown-Halmos type theorems for and factorization results for DTTOs, characterizations and structural theory, essential commutation and related spectral problems (see \cite{C Gu, LI-SANG-DING,SANG-QIN-DING}).

Among the various operator classes arising in Hilbert space operator theory, partial isometries play a crucial role. They naturally appear in the study of shift operators, polar decompositions, operator models, and $C^*$-algebras. Consequently, determining when a truncated or dual truncated Toeplitz operator is a partial isometry becomes an interesting and natural problem. The concrete structure of partially isometric Toeplitz operators were given in the classical Brown and Douglas \cite{Brown_Douglas} classification of partially isometric Toeplitz operators on $H^2$. Later, Deepak, Pradhan, and Sarkar \cite{KD} established the Brown-Douglas result for partially isometric Toeplitz operators over the polydisk $\D^n$. Recently, Debnath, Pradhan, and Sarkar \cite{DEBNATH-PRADHAN-SARKAR} characterized the partially isometric truncated Toeplitz operatoprs with inner symbols on model spaces over the polydisk $\D^n$. But the characterization of partial isometric truncated Toeplitz operators for any bounded symbols is still not known. In this paper, we aim to characterize partially isometric dual truncated Toeplitz operators for some specific symbols which will surely enlarge the earlier classes. More precisely, we give characterization results for truncated and dual truncated Toeplitz operators with symbols of the form $\Bar{u}v$, where $u$ and $v$ are inner functions such that $v$ divides $\theta$.

The structure of the rest of the paper is as follows: In Section \ref{sec2}, we give some preliminary notations and definitions. In Section \ref{sec3}, we provide the characterization of dual truncated Toeplitz operators (and truncated Toeplitz operators) for some specific symbols. Section \ref{sec4} focuses on some interesting facts about the sets of extremal vectors for the non-zero partially isometric truncated and dual truncated Toeplitz operators, which are of independent interest. We conclude with some remarks in Section 5.

\section{Preliminaries}\label{sec2}
	
Let $\clh$ be a complex Hilbert space and let $\clb(\clh)$ be the space of all bounded operators on $\clh$. For $T \in \clb(\clh)$, let $\cln(T)$ and $\clr(T)$ denote the null space and range space of $T$, respectively. An operator $T \in \clb(\clh)$ is said to be a partial isometry if $\|T h\|=\|h\|$ for all $h \in \cln(T)^\perp$. The following characterization of a partial isometry are well-known to the literature (see \cite[Proposition 4.38]{RD}).
	
\begin{lem} \label{Lem:partial isometry equivalences}
Let $T\in \clb(\clh)$. Then the following are equivalent:
\begin{enumerate}
			\item $T$ is a partial isometry.
			\item $TT^*T=T$.
			\item $T^*TT^*=T^*$.
			\item $TT^*$ is a projection.
			\item $T^*T$ is a projection.
			\item $T^*$ is a partial isometry.
\end{enumerate}
\end{lem}

Let $\D$ and $\T$ be the open unit disc and the unit circle in the complex plane $\C$, respectively. Let $L^2(\T)$ be the space of all square integrable functions with respect to normalised Lebesgue measure on $\T$ and $L^\infty(\T)$ denote the Banach algebra of all essentially bounded functions on $\T$. The \emph{Hardy-Hilbert space}, denoted by $\h$, consists of all analytic functions on $\D$ with square summable Taylor coefficients at the origin. Let $H^\infty$ denote the space of all bounded analytic functions on $\D$.
We often identify $H^2$ as a closed subspace of $L^2(\T)$, and write $L^2(\T)= H^2 \oplus {H^2}^{\perp}$. For any $\phi \in L^\infty(\T)$, the Laurent operator on $L^2(\T)$, denoted by $M_\phi$, is defined as
	\[
	M_\phi(f)=\phi f \quad (f \in L^2(\T)),
	\]
and the \textit{Toeplitz operator} $T_\phi: \h \rightarrow \h$ with symbol $\phi$ is defined as
	\[
	T_\phi f=P(\phi f) \quad (f \in \h),
	\] 
where $P$ is the orthogonal projection of $L^2(\T)$ onto $\h{}$. The \textit{Hankel operator} $H_\phi :\h \rightarrow \h{}^\perp$ is defined by
	\[
	H_\phi f= (I-P)(\phi f) \quad ( f\in \h).
	\]
Given a non-constant inner function $\theta$, the \textit{model space} $K_{\theta}$ is defined as $K_\theta = \h \ominus \theta \h$ and clearly, it is $T_z^*$-invariant. For any $\phi \in L^\infty(\T)$, the \textit{truncated Toeplitz operator}, first introduced by Sarason \cite{DS}, is defined by
	\[
	A_\phi f=P_\theta(\phi f), \quad (f \in K_\theta),
	\]
where $P_\theta$ is the orthogonal projection of $L^2(\T)$ onto $K_\theta$. It is easy to see that $\|A_\phi\| \leq \|\phi\|_\infty$. If $\phi(z)=z$, then $A_\phi$ is called the \textit{compressed shift}, denoted by $S_{\theta}$. For any $\phi \in L^\infty(\T)$ the \textit{dual truncated Toeplitz operator}, denoted by $D_\phi$, is defined on $K_{\theta}^{\perp} = L^2(\T) \ominus K_{\theta}$ as
	\[
	{D_\phi}f= Q_\theta(\phi f) \quad (f \in K_\theta^\perp),
	\]
where $Q_\theta$ is the orthogonal projection of $L^2(\T)$ onto $K_\theta^\perp$.  
Clearly,
	\[
	Q_\theta=I-P_\theta=I-P+ M_\theta PM_{\overline{\theta}}.
	\]
In the special case $\phi(z)=z$, $D_\phi$ is referred to as \textit{dual of the compressed shift $S_\theta$}, denoted as $D_\theta$. Also, for $\phi \in L^\infty(\T)$, the \textit{truncated Hankel operator} $B_\phi: K_\theta \rightarrow K_\theta^\perp$ and the \textit{dual truncated Hankel operator} $C_\phi: K_\theta^\perp \rightarrow K_\theta$ are defined as
	\[
	B_\phi f= Q_\theta(\phi f), \quad (f \in K_\theta)
	\]
and
	\[
	C_\phi f=P_\theta(\phi f), \quad (f \in K_\theta^\perp),
	\]
respectively. It is a routine check that $A_\phi^*=A_{\bar{\phi}}$, $D_{\phi}^*=D_{\bar{\phi}}$, and $B_\phi^*=C_{\bar{\phi}}$.	
Thus, the multiplication operator $M_{\phi}$ has the following block matrix representation with respect to the decomposition $L^2(\T)=K_{\theta}\oplus K_{\theta}^{\perp}$:
	\[
	M_{\phi}=
	\begin{bmatrix}
		A_{\phi} & C_{\phi}\\
		B_{\phi} & D_{\phi}
	\end{bmatrix}.
	\]
Now for $\phi, \psi\in L^{\infty}(\T)$, the relation $M_{\phi}M_{\psi}=M_{\phi\psi}$ yields the following identities (see \cite{XD}):
\begin{align}
		A_{\phi\psi}-A_{\phi}A_{\psi} &= B_{\Bar{\phi}}^*B_{\psi}\label{Eq:relation between A and B}\\
		B_{\phi\psi}-D_{\phi}B_{\psi} &= B_{\phi}A_{\psi} \label{Eq:relation between A B and D}\\
		D_{\phi\psi}-D_{\phi}D_{\psi} &= B_{\phi}B_{\Bar{\psi}}^* \label{Eq:relation between D and B}.	
\end{align}

A \textit{conjugation} on a complex Hilbert space is an antilinear isometric involution. In what follows, the operator $C: L^2(\T) \rightarrow L^2(\T)$ defined by
	\[
	Cf=\theta \overline{zf}, \quad (f \in L^2(\T))
	\]
is a conjugation which bijectively maps $\theta \h$ to $\h{}^\perp$ and $K_\theta$ to itself \cite{GM}. We will write $\tilde{f}$ for $Cf$ for the sake of brevity. Note that the TTOs and DTTOs are $C$-symmetric, that is 
	\[
	A_\phi^*=CA_\phi C \quad \text{and} \quad D_\phi^*= C D_\phi C.
	\]
	
In the following theorem, we recall some basic properties of $A_{\phi}, D_{\phi},$ and $B_{\phi}$.
	
\begin{prop}[\cite{XD}, \cite{KLN}]
Let $\phi \in L^\infty(\T)$. The following statements hold:
\begin{enumerate}
			\item $D_\phi$ is a bounded operator and $\|D_\phi\|=\|\phi\|_\infty$.
			\item $A_\phi=0$ if and only if $\phi=\theta \h +\overline{\theta \h}$.
			\item $D_\phi$ is compact if and only if $\phi=0$ a.e.\ on $\T$.
			\item $B_\phi=0$ if and only if $\phi$ is constant.	
		\end{enumerate}
\end{prop}

For $\alpha \in \C$, the Sedlock class $\mathscr{B}_{\alpha}^{\theta}$ is the set of such bounded TTOs which have a symbol of the form $\phi +\overline{\alpha S_\theta\tilde{\phi}} +c$, where $ \phi \in K_\theta$ and $c \in \C$. While $\mathscr{B}_\infty^{\theta}$ is the set of TTOs with co-analytic symbols. The set of all TTOs on $K_\theta$ forms a weakly closed subspace of $\clb(K_\theta)$ that we will denote by $\mathscr{T}_\theta$. We also use the following results related to the product of TTOs given by Sedlock in \cite{NS} to prove our main results. 

\begin{lem}\label{prod}
Let $A_\phi, A_\psi \in \mathscr{T}_{\theta}$ such that $A_\phi A_\psi \in \mathscr{T}_{\theta}$. If one of the operators in the product is in $\mathscr{B}_{\alpha}^{\theta}$ for $\alpha \in \C \cup\{\infty\}$, then either it is a constant multiple of identity or the other is also in $\mathscr{B}_{\alpha}^{\theta}$.
\end{lem}

\begin{lem}\label{typealphachar}
Let $A \in \mathscr{B}_{\alpha}^\theta$ for $\alpha \in \D$. There exists a function $ \psi \in H^\infty$ such that $\|\psi\|_\infty=\|A\|$ and $A=A_{\frac{\psi}{1-\alpha \bar{\theta}}}$. Moreover, if $\phi, \psi \in H^\infty$, then $A_{\frac{\phi}{1-\alpha \bar{\theta}}}A_{\frac{\psi}{1-\alpha \bar{\theta}}}=A_{\frac{\phi\psi}{1-\alpha \bar{\theta}}}$.
\end{lem}

\section{Partially isometric dual truncated Toeplitz operators}\label{sec3}

In this section we obtain the characterization of partially isometric DTTO (hence TTO)
for the symbol of the form $\bar{u}v$, where $u, v$ are inner functions and $v$ divides 
$\theta$.

Before proceeding to the characterization result, we first show the equivalent relation of TTO, DTTO and truncated Hankel operator.
	
\begin{lem}\label{Lem:A B and D partial isometry equivalences}
Let $\phi\in L^{\infty}(\T)$ with $|\phi|=1$ a.e.\ on $\T$. Then the following are equivalent:
		\begin{enumerate}
			\item $A_{\phi}$ is a partial isometry.
			\item $B_{\phi}$ is a partial isometry.
			\item $D_{\phi}$ is a partial isometry.
		\end{enumerate}
\end{lem}

\begin{proof}
Suppose that $\phi\in L^{\infty}(\T)$ with $|\phi|=1$ a.e.\ on $\T$. Now the relations in \eqref{Eq:relation between A and B} and \eqref{Eq:relation between D and B} give
\begin{align*}
I_{K_{\theta}}-A_{\phi}^*A_{\phi} = A_{\Bar{\phi}\phi}-A_{\Bar{\phi}}A_{\phi} &= B_{\phi}^*B_{\phi}\\
I_{K_{\theta}^{\perp}}-D_{\phi}D_{\phi}^* = D_{\phi\Bar{\phi}}-D_{\phi}D_{\Bar{\phi}} &= B_{\phi}B_{\phi}^*.
\end{align*}
It now follows from Lemma \ref*{Lem:partial isometry equivalences} that $A_{\phi}$ is a partial isometry if and only it $B_{\phi}$ is a partial isometry if and only if $D_{\phi}$ is a partial isometry.
\end{proof}

The partially isometric TTOs with inner symbols were recently characterized by Debnath, Pradhan and Sarkar in \cite{DEBNATH-PRADHAN-SARKAR}. Indeed, they proved the following result: 
\begin{thm}
Let $u$ and $\theta$ be non-constant inner functions in $H^\infty$. Then $A_u$ on $K_\theta$ is a partial isometry if and only if $u$ divides $\theta$ or $\theta$ divides $u$.
\end{thm}

However, the following theorem provides a new proof of this characterization by employing dual truncated Toeplitz operators. Our approach is more algebraic in nature and uses the structural properties of DTTOs. It also reveals an interesting interplay between TTOs and DTTOs.

\begin{thm}\label{Thm:Du partial isometry characterization}
Let $\theta$ be a non-constant inner function and $D_u$ be the dual TTO on $K_{\theta}^{\perp}$ corresponding to the inner function $u$. Then the following are equivalent:
\begin{enumerate}
\item $D_u$ is a partial isometry.
\item $u$ divides $\theta$ or $\theta$ divides $u$.
\end{enumerate}
\end{thm}

\begin{proof}
Note that $K_{\theta}^{\perp} = \theta H^2 \oplus {H^{2}}^{\perp}$. Now for any $h \in H^2$
we have
\begin{align*}
D_u^*D_uD_u^*(\theta h) &= D_u^*D_u(M_\theta P M_{\Bar{\theta}}(\Bar{u}\theta h)+(I-P)(\Bar{u}\theta h))\\
			&= D_u^*D_u(M_{\theta} P(\Bar{u}h) + (I-P)(\Bar{u}\theta h)\\
			&= D_u^*D_u(\theta T_{u}^*h+\Bar{u}\theta h-P(\Bar{u}\theta h))\\
			&= D_u^*(M_{\theta} PM_{\Bar{\theta}}(u\theta T_{u}^*h+\theta h-uP(\Bar{u}\theta h))+(I-P)(u\theta T_{u}^*h+\theta h-uP(\Bar{u}\theta h))\\
			&= D_u^*(M_{\theta} P(uT_{u}^*h+h-u\Bar{\theta}P(\Bar{u}\theta h)))\\
			&= D_u^*(u\theta T_{u}^*h+\theta h-\theta T_{\Bar{u}\theta}^*T_{\Bar{u}\theta}h)\\
			&= M_{\theta}PM_{\Bar{\theta}}(\theta T_{u}^*h+\Bar{u}\theta h-\Bar{u}\theta T_{\Bar{u}\theta}^*T_{\Bar{u}\theta}h)+(I-P)(\theta T_{u}^*h+\Bar{u}\theta h-\Bar{u}\theta T_{\Bar{u}\theta}^*T_{\Bar{u}\theta}h)\\
			&= \theta T_{u}^*h+\theta T_{u}^{*}h-\theta T_{u}^*T_{\Bar{u}\theta}^*T_{\Bar{u}\theta}h + (I-P)(\Bar{u}\theta h-\Bar{u}\theta T_{\Bar{u}\theta}^*T_{\Bar{u}\theta}h)\\
			&= \theta T_{u}^*h+(I-P)(\Bar{u}\theta h-\Bar{u}\theta T_{\Bar{u}\theta}^*T_{\Bar{u}\theta}h).
\end{align*}
Also, for $g\in H^2$ we have 
\[
D_u^*D_uD_u^*(\overline{zg})=D_u^*D_uQ_{\theta}(\overline{zug})=D_u^*D_u(\overline{zug})=D_u^*Q_{\theta}(\overline{zg})=D_u^*(\overline{zg}).
\]
Thus, $D_u$ is a partial isometry if and only if $D_u^*D_uD_u^*f=D_u^*f$ for all $f\in K_{\theta}^{\perp}$ if and only if $D_u^*D_uD_u^*(\theta h)=D_u^*(\theta h)$ for all $h \in H^2$. 
That is,
\[
\theta T_{u}^*h+(I-P)(\Bar{u}\theta h-\Bar{u}\theta T_{\Bar{u}\theta}^*T_{\Bar{u}\theta}h)=\theta T_{u}^*h+(I-P)(\Bar{u}\theta h) \quad (h \in H^2).
\]
Therefore, for all $h\in H^2$
\begin{align*}
(I-P)(\Bar{u}\theta T_{\Bar{u}\theta}^*T_{\Bar{u}\theta}h) &=0\\
			\Bar{u}\theta T_{\Bar{u}\theta}^*T_{\Bar{u}\theta}h&=P(\Bar{u}\theta T_{\Bar{u}\theta}^*T_{\Bar{u}\theta}h)\\
			T_{\Bar{u}\theta}^*T_{\Bar{u}\theta}h &= u\Bar{\theta}T_{\Bar{u}\theta} T_{\Bar{u}\theta}^*T_{\Bar{u}\theta}h\\
			T_{\Bar{u}\theta}^*T_{\Bar{u}\theta}h &= T_{\Bar{u}\theta}^*T_{\Bar{u}\theta}T_{\Bar{u}\theta}^*T_{\Bar{u}\theta}h=(T_{\Bar{u}\theta}^*T_{\Bar{u}\theta})^2h.
\end{align*}
Hence $T_{\Bar{u}\theta}^*T_{\Bar{u}\theta}$ is a projection operator on $H^2$. Consequently, by Lemma \ref{Lem:partial isometry equivalences}, $T_{\Bar{u}\theta}$ is a partial isometry on  $H^2$. Therefore, it follows from Brown-Douglas \cite{Brown_Douglas} characterization of partially isometric Toeplitz operators that either $\Bar{u}\theta$ or $\Bar{\theta}u$ is inner function. That is, either $u | \theta$ or $\theta | u$.
This completes the proof.
\end{proof}

Next we will give a representation for symbols $\phi$ of partially isometric TTO (or DTTO) with $\|\phi \|_{\infty} \leq 1$. This may be known to the experts but here we give a proof for the sake of completeness. Before that, we are stating a lemma which will be useful in the proof. 
 	
\begin{lem}\cite[Corollary 3, pp.\ 12]{NIK}.\label{Lem:inner outer factorization}
Let $f\in H^2$. Then the following are equivalent:
		\begin{enumerate}
			\item $f$ is an outer function.
			\item If $g\in H^2, g/f \in L^2(\T)$ then $g/f \in H^2.$
\end{enumerate}
\end{lem}

\begin{lem}\label{Characterisation}
Let $\phi\in L^{\infty}(\T)$ with $\|\phi\|_{\infty}\leq 1$. Suppose $A_{\phi}$ is a non-zero partially isometric truncated Toeplitz operator on $K_{\theta}$. Then 
$\phi=\overline{u}v$ for some inner functions $u$ and $v$ in $H^{\infty}$. In particular, if $\phi\in H^{\infty}$ with $\|\phi\|_{\infty}\leq 1$, then $\phi$ is inner.
\end{lem}

\begin{proof}
Since $A_\phi$ is non-zero, there exists $0 \neq f \in K_{\theta}\ominus \cln(A_\phi)$ such that $\|A_{\phi}f\|=\|f\|.$ Note that
\[
\|f\|=\|A_\phi f\|=\|P_\theta(\phi f)\|\leq \|\phi f\|\leq \|\phi\|_\infty \|f\|\leq \|f\|.
\]
Thus 
\[
\|f\|=\|\phi f\|=\|P_{\theta}(\phi f)\|.
\]		
By a well-known theorem of the Riesz brothers (see \cite[Chapter 6]{RD}), $f\neq 0$ a.e.\ on $\T$, so that $\|f\|=\|\phi f\|$ implies $|\phi|=1$ a.e.\ on $\T$. Moreover, 
\[
\|\phi f\|=\|P_\theta(\phi f)\| \iff \phi f=P_{\theta}(\phi f).
\]
Therefore, $\phi f\in K_{\theta}$. Let $f=f^{i}f^{o}$ be the inner-outer factorization of $f$. Now Lemma \ref{Lem:inner outer factorization} infers that $\phi f^i$ is inner in $H^2$, say $v$, and take $u=f^i$ then $\phi=\overline{u}\,v$, where $u$ and $v$ are inner functions. 

Clearly, if $\phi\in H^{\infty}$ with $\|\phi\|_{\infty}\leq 1$, then from the above $|\phi|=1$ a.e.\ on $\T$ and hence $\phi$ is inner. 
\end{proof}

In the following theorem, we characterize the partially isometric DTTO (and hence TTO) $D_\phi$ where $\phi=\bar{u}\theta$ for some inner function $u$. Note that if $\bar{u}\theta$ is co-analytic, then Theorem \ref{Thm:Du partial isometry characterization} will provide the required characterization. So for the following theorem, we will assume that $\bar{u}\theta$ is not co-analytic.

\begin{thm}
Let $\theta, u$ be non-constant inner functions such that $\bar{u}\theta$ is not co-analytic. Then $D_{\bar{u} \theta}$ is a partial isometry if and only if $u$ divides $\theta$ or $\theta^2$ divides $u-c$ for some constant c. 
\end{thm}

\begin{proof}
First note that if $D_{\bar{u}\theta}$ is a partial isometry, then $A_{\bar{u}\theta}$ is a partial isometry by Lemma \ref{Lem:A B and D partial isometry equivalences}. Now if $A_{\bar{u}\theta} = 0$, then there exist $h_1, h_2 \in H^2$ such that
$\bar{u}\theta= \theta h_1 +\overline{\theta h_2}$.
Thus
\[
{u}=\overline{h_1}+{\theta^2 h_2}
\]
which implies that $h_1$ is a constant, say $c$. Therefore, $\theta^2 |(u-c)$.
		
Now if $A_{\bar{u} \theta}$ is a non-zero partial isometry, then there exists $0 \neq f \in \cln(A_{\bar{u}\theta})^{\perp}$ such that $\|A_{\bar{u}\theta} \| = \|f \|$.
This implies that $\bar{u}\theta f \in K_\theta$ and hence
\[
T_\theta^* (\bar{u}\theta f)=0.
\]
Now
\[
0=T_\theta^* (\bar{u}\theta f)= T_u^* f, \quad \mbox{and hence} \quad A_{\bar{u}} f =0.
\]
Thus, $\cln (A_{\bar{u}\theta})^\perp \subseteq \cln (A_{\bar{u}})$ which implies 	$\clr(A_u) \subseteq \cln(A_{\bar{u}\theta})$. Consequently, $A_{\bar{u}\theta} A_u=0$. Note that $A_{u}$ is of type $0$. So, by Lemma \ref{prod}, there may have three cases:

\NI		
\textit{Case I:} $A_{u}=0$.

It follows that $u\in\theta \h$ but this  will imply that $\bar{u}\theta$ is co-analytic which is not the case.
\vspace{.1cm}
	
\NI		
\textit{Case II:} $A_{u}= \alpha I$ for some $\alpha \neq 0$.

In this case, we get $A_{\bar{u}\theta}=0$ which is not so.
\vspace{.1cm}

\NI		
\textit{Case III:} $A_{\bar{u}\theta}$ is of type $0$.

It follows by Lemma \ref{typealphachar} that $A_{\bar{u}\theta}=A_\psi$ for some $\psi \in H^\infty$ such that $\|\psi\|_\infty=\|A_{\bar{u}\theta}\|=1$. Thus $A_{\psi}$ is a non-zero partial isometry and hence $\psi$ is inner. Now Theorem \ref{Thm:Du partial isometry characterization} yields either $\theta | \psi$ or $\psi | \theta$. If $\theta | \psi$, then clearly $A_\psi=0$ which is not possible as $A_{\bar{u}\theta}$ is non-zero. Therefore, $\psi | \theta$, i.e., $\theta=\psi w$ for some inner function $w$.
Again $A_{\bar{u}\theta}=A_\psi$ implies that there exist $h_1, h_2 \in \h$ such that
\[
\bar{u}\theta =\psi +\theta h_1 +\overline{\theta h_2}.
\]
Therefore,
\begin{align} \label{eq1}
u &= \bar{\psi}\theta +\overline{h_1} +\theta^2 h_2.
\end{align}
Since $u, \bar{\psi}\theta,$ and $\theta ^2 h_2 \in \h$, $\overline{h_1} \in \h$. 
Hence $h_1$ is constant, say $\beta$. Now (\ref{eq1}) implies 
\begin{align}\label{eq2}
u\psi=\theta + \bar{\beta} \psi +\theta^2 \psi h_2.
\end{align}
Note that 
\[
A_{\psi u}=A_\psi A_u=A_{\bar{u}\theta}A_u=0. 
\]
Then considering the corresponding TTO in (\ref{eq2}) we get $\bar{\beta} A_\psi=0$. Consequently, $\beta=0$. Now from the equation (\ref{eq1}), we get $u=w +\theta^2 h_2$. Again considering the corresponding TTO, we have $A_u=A_w$. Since $w | \theta$, $A_w$ is a partial isometry and hence $A_u$ is so. Also $A_u$ is non-zero, thus $u | \theta$.

Conversely, if $\theta^2|u-c$ for some constant $c$, then $A_{\bar{u}\theta}=0$ and hence a partial isometry. If $u|\theta$, then $\bar{u}\theta$ is inner and also divides 
$\theta$. Therefore, by Theorem \ref{Thm:Du partial isometry characterization},  $D_{\bar{u}\theta}$ is a partial isometry.
This completes the proof.		
\end{proof}

\begin{cor}
If $A_{\bar{u}\theta}$ is a non-zero partial isometry, then $\bar{u}\theta$ is either analytic or co-analytic.
\end{cor}

Recall that if $u$ and $v$ are inner functions, then $K_{u}\cap K_{v}=\{0\}$ if and only if $gcd(u,v)=1$ (cf. \cite{GM}). Thus we get the following corollary:
	
\begin{cor}
Let $\theta$ and $u$ be inner functions such that $gcd(u,\theta)=1$. Then $D_{\bar{u}\theta}$ is a partial isometry if and only if $A_{\bar{u}\theta} = 0$.
\end{cor}

Next we will characterize partially isometric DTTO with symbol $\bar{u}v$, where $v$ divides $\theta$. Note that while considering the symbols of the form $\bar{u}v$, where $u$ and $v$ are arbitrary inner functions, we can assume that $gcd(u,v)=1$ until it does not hamper any given condition.
	
\begin{lem}\label{lemma for zero case}
Let $u, v$ and $\theta$ be inner functions with $gcd(u,v)=1$ such that $v$ divides 
$\theta$. If $A_{\bar{u}v} = 0$, then either $\theta=cv$ where $c$ is a unimodular constant or $v$ is a unimodular constant.
\end{lem}

\begin{proof}
Note that $A_{\bar{u}v} = 0$ if and only if there exists $h_1, h_2 \in H^2$ 
such that $\bar{u}v =\theta h_1 +\overline{\theta h_2}$. Thus
\begin{align}\label{div}
\bar{u}=\bar{v} \theta h_1 + \overline{ v\theta h_2},
\end{align}
which implies that $\bar{v}\theta h_1 \in \overline{H^2}$. But since $\bar{v}\theta$ is inner, $\bar{v}\theta h_1$ must be constant, and hence $h_1$ is constant. Now if $h_1$ 
is non-zero, then $\theta=cv$, where $c$ is a unimodular constant. If $h_1=0$, then the equation $(\ref{div})$ implies that $u=v \theta h_2$, i.e., $v$ divides $u$. Since $gcd (u,v)=1$, $v$ must be a unimodular constant. 
\end{proof}

\begin{thm}\label{vdividestheta}
Let $\theta$ be a non-constant inner function and $\phi=\Bar{u}v$ not co-analytic, where $u$ and $v$ are inner functions such that $gcd(u,v)=1$. Suppose $v$ is non-constant and divides $\theta$ non-trivially, then the following are equivalent:
\begin{enumerate}
			\item $D_{\phi}$ is a partial isometry.
			\item $u$ is a unimodular constant.			
\end{enumerate}  
\end{thm}

\begin{proof} $(1) \Rightarrow (2):$
Assume that $D_{\phi}$ is a partial isometry, where $\phi=\Bar{u}v$ and also $v$ divides $\theta$. Consider the corresponding truncated Toeplitz operator $A_{\phi}$ on the model space $K_{\theta}$. Then $A_{\phi}$ is also a partial isometry by Lemma \ref{Lem:A B and D partial isometry equivalences}. If $A_{\bar{u}v}= 0$, then by Lemma \ref{lemma for zero case} and given assumption, $v$ is a unimodular constant. In that case $\phi=\bar{u}$ which is co-analytic. 

On the other hand, if $A_{\bar{u}v} \neq 0$, then there exists $ 0 \neq f\in \cln(A_{\Bar{u}v})^{\perp}$ such that $\| A_{\bar{u}v} f\| = \|f \|$.
This yields $\Bar{u}vf\in K_{\theta}$. Therefore,
\[
P(\Bar{u}v\Bar{\theta}f) = T_{\theta}^*(\Bar{u}vf)= 0.
\] 
Thus
\[
A_{\Bar{u}v\Bar{\theta}}f = P_{\theta}(\Bar{u}v\Bar{\theta}f)= P_{\theta}P(\Bar{u}v\Bar{\theta}f) =0,
\]
and hence $\cln(A_{\Bar{u}v})^{\perp} \subseteq \cln(A_{\Bar{u}v\Bar{\theta}})$.  Consequently,
\[
\clr(A_{u\Bar{v}\theta})\subseteq \overline{\clr(A_{u\Bar{v}\theta})}=\cln(A_{\Bar{u}v\Bar{\theta}})^{\perp}\subseteq \cln(A_{\Bar{u}v}),
\]
which implies 
\[
A_{\phi}A_{u\Bar{v}\theta} = A_{\Bar{u}v}A_{u\Bar{v}\theta}=0.
\]
Now $A_{u\bar{v}\theta}$ is of type $0$. Then, by Lemma \ref{prod}, either $A_{u\bar{v}\theta} = cI$ for some constant $c$ or $A_\phi$ is of type $0$. If $A_{u\bar{v}\theta}=cI$, clearly $A_\phi = 0$ which is not the case.  

Also, if $A_{u\bar{v}\theta} = 0$, then it is easy to check that $v$ divides $u$ but since $gcd(u,v)$ is constant, $v$ must be constant. Thus $A_\phi$ is of type $0$. It follows that there exists $\psi \in H^\infty$ such that $A_\phi=A_\psi$ and $\|\psi\| =\|A_\phi\|=1$.  Again $A_\phi$ is a non-zero partial isometry, so is $A_\psi$. Then $\psi | \theta$. Consequently, there exist $h_1, h_2 \in \h$ such that
\[
\bar{u}v=\psi+\theta h_1 +\overline{\theta h_2}.
\] 
Therefore,		
\[
u\bar{v}\theta=\bar{\psi}{\theta} +\overline{h_1} + \theta^2 h_2,
\]
which implies $h_1=0$. Then
\[
A_{u\bar{v}\theta}= A_{\bar{\psi}{\theta}}.
\]
Note that $\theta=(\bar{\psi}{\theta}) \psi$. Thus, $A_{\bar{\psi}{\theta}}$ is a partial isometry and so is $A_{u\bar{v}\theta}$. Hence, by Lemma \ref{Lem:A B and D partial isometry equivalences} and Theorem \ref{Thm:Du partial isometry characterization}, $u\bar{v}\theta |\theta$, i.e., $\theta=u\bar{v}\theta w$ for some inner function $w$. 
It follows that $v =uw$ and hence $u|v$. Now $gcd(u,v)=1$ infers that u is a unimodular constant.

The converse part follows trivially from Theorem \ref{Thm:Du partial isometry characterization}.
\end{proof}

\section{Extremal vectors}\label{sec4}
	
In this section, we will examine the structure of initial spaces of partially isometric TTOs and DTTOs. We will be needed the following definition for truncated Hankel operator given in \cite{PM}.
	
\begin{lem}
For $\phi \in L^\infty(\T)$ and $f \in K_\theta$,
\[
B_\phi f=H_\phi H_{\bar{\theta}}^*H_{\bar{\theta}}f+M_\theta H_{\bar{\phi}}^*H_{\bar{\theta}}f.
\]
\end{lem}
	
We will start with the initial space of $A_{\bar{u}v}$ on $K_\theta$, where $u,v$ are inner functions.

\begin{prop}
Let $u$ and $v$ be non-constant inner functions such that $A_{\bar{u}v}$ is a non-zero partial isometry on $K_\theta$. Then
\[
\cln(A_{\bar{u}v})^\perp=\{f \in K_\theta \cap u\, \overline{gcd(u,v)}\h: vf \in K_{u\theta}\}.
\]
\end{prop}

\begin{proof}
Note that for $\|\phi\|_{\infty} \leq 1$, if $A_\phi$ is a partial isometry, then 
$\cln(A_\phi)^\perp=\cln(B_\phi)$. So, we will find the structure of $\cln(B_{\bar{u}v})$.

Let $f \in \cln(B_{\bar{u}v})$. Then
\[
H_{\bar{u}v}H_{\bar{\theta}}^*H_{\bar{\theta}}f+ M_\theta H_{u\bar{v}}^*H_{\bar{\theta}}f = B_{\bar{u}v}f=0.
\]
It follows that
\begin{align*}
 & H_{\bar{u}v}H_{\bar{\theta}}^*H_{\bar{\theta}}f+M_\theta H_{u\bar{v}}^*H_{\bar{\theta}}f  =0\\
			\Leftrightarrow \,\,	& H_{\bar{u}v}H_{\bar{\theta}}^*H_{\bar{\theta}}f=0 \text{ and } M_\theta H_{u\bar{v}}^*H_{\bar{\theta}}f =0\\
			\Leftrightarrow \,\,& H_{\bar{u}v} (I-T_\theta T_\theta^*) f \text{ and } (T_{\bar{u}v\bar{\theta}}- T_{\bar{u}v}T_\theta^*)f=0\\
			\Leftrightarrow \,\,& H_{\bar{u}v}f=0 \text{ and } T_{\bar{u}v\bar{\theta}}f=0,
\end{align*}
where the last implication holds as $f \in K_\theta$. Let $gcd(u,v)=w$ with $u=ww_1$ and $v=ww_2$. Then $\bar{u}v=\bar{w_1}w_2$ and $gcd(w_1,w_2)=1$. Now $H_{\bar{u}v}f=0$ implies $H_{\bar{w_1}w_2}f=0$. Thus 
\[
f \in w_1\h=uw\h.
\]
This completes the proof.
\end{proof}

\begin{cor}
Let $u$ and $v$ be non-constant inner functions with $gcd(u,v)=1$ such that $A_{\bar{u}v}$ is a non-zero partial isometry on $K_\theta$. Then
\begin{align}\label{gcd1}
\cln(A_{\bar{u}v})^\perp=\{f \in K_\theta \cap u \h: vf \in K_{u\theta}\}.
\end{align}
\end{cor}

\begin{rem}
For $T\in \clb(\clh)$, define 
\[
\cle_{T}=\{h \in \clh: \|Th\|=\|T\|\|h\|\},
\]
that is, $\cle_{T}$ is the set of all such points on which $T$ attains its norm, called the set of extremal vectors. From the above proof, for a non-zero partially isometric 
TTO $A_{\bar{u}v}$ on $K_\theta$, 
\[
\cle_{A_{\bar{u}v}}=\{f \in K_\theta \cap u \overline{gcd(u,v)}\h: vf \in K_{u\theta}\}.
\]
\end{rem}

Next we will discuss about the initial space of $D_{\bar{u}v}$ whenever $A_{\bar{u}v}$ is a non-zero partial isometry. Without loss of generality, we can assume that $gcd(u,v)=1$.
Recall that
\begin{align*}
		\cln(D_{\bar{u}v})^\perp&=\{f \in K_\theta^\perp: \|D_{\bar{u}v}f\|=\|f\|\}\\
		&=\{f \in K_\theta^\perp: \|Q_{\theta}(\bar{u}v f)\|= \|f \| \}\\
		&=\{f \in K_\theta^\perp: \bar{u}v f \in K_{\theta}^\perp\}.
\end{align*}
Since $K_{\theta}^{\perp}= \theta \h \oplus \h{}^{\perp}$, we will define the following sets:
\[
\cle_{+}=\{ \theta h \in \theta \h : \|D_{\bar{u}v}(\theta h)\|=\|\theta h\|\},
\]
and
\[
\cle_{-}=\{\overline{zg} \in \h{}^\perp : \|D_{\bar{u}v}(\overline{zg})\|=\|\overline{zg}\|\}.
\]
It is easy to verify that $\theta u\h \subset \cle_+$ and $\overline{zv\h} \subset \cle_-$. Thus
	\[
	\theta u\h \oplus \overline{zv\h} \subseteq \cle_{D_{\bar{u}v}}.
	\] 

We claim that if $A_{\bar{u}v}$ is a non-zero partial isometry (so is $D_{\bar{u}v}$), then $\cle_{+}$ and $\cle_{-}$ are subspaces and also
\[
\cle_{+}=\cle_{D_{\bar{u}v}} \cap \theta \h =\theta u\h \quad \text{ and } \quad 
\cle_{-}=\cle_{D_{\bar{u}v}} \cap \h{}^\perp=\overline{zv\h}.
\]

First observe that if $A_{\bar{u}v}$ is non-zero partial isometry, then (\ref{gcd1}) implies that $K_\theta \cap u\h \neq \{0\}$ and hence $\cln(T_{\bar{\theta}u}) \neq \{0\}$ (see \cite{XY}). By Coburn's alternative, $\cln(T_{\bar{u}\theta})=\{0\}$, that is, $K_u \cap \theta \h=\{0\}$. Now assume that $\theta h \in \cle_{D_{\bar{u}v}} \cap \theta \h =\cle_+$. Then $\|D_{\bar{u}v}(\theta h)\|=\|\theta h\|$ which implies $\bar{u}v \theta h \in K_\theta^\perp$. Thus there exist $h_1, h_2 \in \h$ such that
\[
\bar{u}v\theta h=\theta h_1 +\overline{zh_2}.
\]
Therefore,
\begin{align}\label{uH^2}
v\theta h =u\theta h_1 +u\overline{zh_2}.
\end{align}
Since $v\theta h, u\theta h_1 \in \h$, so does $u\overline{zh_2}$. Hence, $u\overline{zh_2} \in K_u \cap \theta \h =\{0\}$. Consequently, $h_2=0$ and from (\ref{uH^2}), 
\[
\bar{u}v h=h_1 \in \h.
\]
Let $h=h^{i}h^{o}$ is the inner-outer factorization of $h \in \h$. Thus, 
$\bar{u}vh^{i}h^{o} \in \h$, and hence, by Lemma \ref{Lem:inner outer factorization}, $\bar{u}vh^{i}$ is inner, that is, $u$ divides $vh^{i}$. Now $gcd(u,v)=1$ yields $u$ divides $h^{i}$. Thus we get $h \in u\h$. It proves the first claim. 

For the second one, assume that $\overline{zg}\in \cle_{D_{\bar{u}v}} \cap \h{}^\perp$, then $\|D_{\bar{u}v} {(\overline{zg})}\|=\|\overline{zg}\|$. Since $D_{\bar{u}v}$ is $C$-symmetric, we have
\begin{align*}
		&\|CD_{\bar{u}v} {(\overline{zg})}\|=\|C(\overline{zg})\|\\
		\implies & \|D_{\bar{v}u} C(\overline{zg})\|=\|C(\overline{zg})\|\\
		\implies & \|D_{\bar{v}u} {(\theta g)}\|=\|\theta g\|.	
\end{align*} 
Since $D_{\bar{u}v}$ is a partial isometry, so is $D_{\bar{v}u}$. Hence by above argument, $g \in v\h$. This proves our claim.

We have shown that $\cle_{+}=\theta uH^2$ and $\cle_{-}=\overline{z v h_2}$ and 
	\[
	\cle_{+} \oplus \cle_{-} \subseteq \cle.
	\]

We will prove the equality in some specific cases.

\begin{thm}\label{NP}
Let $A_v$ be a non-zero partial isometry on $K_{\theta}$ for some inner function $v$. Then $\cle_{D_{v}}=\cle_{+} \oplus \cle_{-}$.
\end{thm}

\begin{proof}		
Since $A_v$ be a non-zero partial isometry, so is $D_v$ and hence, by \ref{Thm:Du partial isometry characterization}, either $v|\theta$ or $\theta |v$. If $\theta |v$, then $A_{v}=0$ which is not the case. Thus, $v|\theta$, that is, $\bar{v}\theta$ is inner. One can easily check that 
		\[
		\cle_{+}=\theta \h \quad \text{ and } \cle_{-}= \overline{zv\h}.
		\]
It is now enough to prove that
		\[
		\cle_{D_{v}}\subseteq \theta\h\oplus \overline{zv\h}.
		\]
Let $ \theta h_1 \oplus \overline{zh_2} \in \cle_{D_{v}}$. Then there exist $g_1,g_2 \in \h$ such that
\begin{align*}
			&v(\theta h_1+\overline{zh_2})=\theta g_1 +\overline{zg_2},\\
\mbox{and hence~}			& \theta h_1+\overline{zh_2}=\bar{v}\theta g_1 +\overline{zvg_2}.
\end{align*} 
Since $h_1, \bar{v}\theta g_1 \in \h$, we get $h_2=vg_2$. This completes the proof.
\end{proof}

\begin{rem}\label{Equality}
Under the assumptions of the previous theorem, using the fact that $D_u$ is a $C$-symmetric operator, we get
		\[
		\cle_{D_{\bar{u}}}=\theta u\h \oplus\h{}^\perp.
		\]
\end{rem}
	
Note that from Theorem \ref{vdividestheta} and Theorem \ref{NP}, we get the following result.
	
\begin{cor}
Let $u,v, \theta$ be inner functions such that $\text{gcd}(u, v)=1$ and $v$ divides $\theta$. If $A_{\bar{u}v}$ is a non-zero partial isometry, then $\cle_{D_{\bar{u}v}}=\cle_{+} \oplus \cle_{-}$.
\end{cor}

The following  theorem is a more stronger result and it is true for any symbol of type 
$\bar{u}v$, where $u$ and $v$ are inner functions.

\begin{thm}
Let $A_{\bar{u}v}$ be a non-zero partial isometry such that $\text{gcd}(u, v)=1$ and $v$ divides $\theta$. Then $\cle_{D_{\bar{u}v}}=\cle_+ \oplus \cle_-$ if and only if either $u$ or $v$ is a unimodular constant.
\end{thm}

\begin{proof}
Suppose that $\cle_{D_{\bar{u}v}}=\cle_+ \oplus \cle_-$. Then
\[
\cln(D_{\bar{u}v})^\perp=\theta u \h \oplus \overline{zv\h}.
\]
Therefore,
\[
\cln(D_{\bar{u}v})=\theta K_u \oplus \overline{zK_v}.
\]
Let $g \in K_u$ and $h \in K_v$, then $\theta g, \overline{zh}\in \cln(D_{\bar{u}v})$. 
It follows that $\theta h =C(\overline{zh}) \in \cln(D_{\bar{v}u})$. Note that $\theta g \in \cln(D_{\bar{u}v})$ which implies $D_{\bar{u}v}(\theta g)=0$. Thus
\[
\bar{u}v \theta g \in K_\theta \implies \ D_{\bar{u}v}(\theta g)=0.
\]
Therefore, $\bar{u}v \theta g \in K_\theta$ and hence 
\[
T_u^* T_vg = T_\theta^*T_u^*T_vT_\theta g =0.
\]
On the similar lines, we can prove that $\theta h =C(\overline{zh}) \in \cln(D_{\bar{v}u})$ which implies that $T_v^*T_uh=0$. But by Coburn's alternative, either $T_u^*T_v$ or $T_v^*T_u$ is injective. Consequently, either $K_u =\{0\}$ or $K_v=\{0\}$.

The converse part follows from Theorem \ref{NP} and Remark \ref{Equality}.		
\end{proof}

\section{Concluding remarks}
	
With the help of the given structure of initial spaces and set of extremal vectors, we may try to characterize partial isometric TTO and DTTO for any general symbol of the form $\bar{u}v$, where $u$ and $v$ are inner functions.
The results obtained in this paper provide a partial solution to the problem of characterizing partially isometric truncated Toeplitz operators and dual truncated Toeplitz operators associated with $L^{\infty}(\T)$-symbols. More precisely, we solved the problem for symbols of the form $\bar{u}v$, where $v$ divides $\theta$.

An important remaining question is whether a complete characterization can be obtained for arbitrary symbols in $L^{\infty}(\T)$ (see \cite{DEBNATH-PRADHAN-SARKAR}). In particular, the unimodular case appears to be especially significant. Since partially isometric TTO or DTTO with unimodular symbol can be expressed in the form $\bar{u}v$, where $u$ and $v$ are inner functions (see Lemma \ref{Characterisation}). The problem essentially reduces to understanding the general $\bar{u}v$-symbol case without divisibility restrictions. The question remains open and seems to require new ideas beyond the methods developed here and will be carried out in the future project.

\label{Ref}


\begin{thebibliography}{99}
		
		\bibitem{BARANOV-BESSONOV-KAPUSTIN}{A. Baranov, R. Bessonov, and V. Kapustin,}
		{\it Symbols of truncated Toeplitz operators,}
		{J. Funct.\ Anal.\ $\bm{261}$ (2011), 3437--3456.}
		
		\bibitem{B-C-F-M-T}{A. Baranov, I. Chalendar, E. Fricain, J. Mashreghi, and D. Timotin,}
		{\it Bounded symbols and reproducing kernel thesis for truncated Toeplitz operators,}
		{J. Funct.\ Anal.\ $\bm{259}$ (2010), 2673--2701.}
		
		\bibitem{RB}{R. V. Bessonov,} 
		{\it Truncated Toeplitz operators of finite rank,} 
		{Proc.\ Amer.\ Math.\ Soc.\ (2014), 1301--1313}.
		
		\bibitem{Brown_Douglas}{A.~Brown, R. G. Douglas,}
		{\it Partially isometric Toeplitz operators,}
		{Proc.\ Amer.\ Math.\ Soc.\ $\bm{16}$ (1965), no.\ 4, 681--682.}
		
		
		\bibitem{ICF}{I. Chalendar, E. Fricain, and D. Timotin,}
		{\it A survey of some recent results on Truncated Toeplitz operators,}
		{In: Recent progress on operator theory and approximation in spaces of analytic functions, Contemp. Math., 679, Amer. Math. Soc. , Providence, RI, (2016), 59--77}.
		
		
		\bibitem{CIMA-GARCIA-ROSS-WOGEN}{J. A. Cima, S. R. Garcia, W. T. Ross, and W. R. Wogen,}
		{\it Truncated Toeplitz operators: spatial isomorphism, unitary equivalence, and similarity,}
		{Indiana Univ. Math. J. $\bm{59}$ (2010), 595--620.}
		
		\bibitem{CIMA-ROSS-WOGEN}{J. A. Cima, W. T. Ross, and W. R. Wogen,}
		{\it Truncated Toeplitz operator on finirte dimensional spaces,}
		{Oper.\ Matrices $\bm{2}$ (2008), 357--369.}
		
		\bibitem{DEBNATH-PRADHAN-SARKAR}{R. Debnath, D. K. Pradhan, and J. Sarkar,}
		{\it Pairs of inner projections and two applications,}
		{J. Funct.\ Anal.\ $\bm{286}$ (2024), no.~2, Paper No.\ 110216, 26 pp.}
		
		\bibitem{KD}{K. D. Deepak, D. Pradhan, and J. Sarkar,}
		{\it Partially isometric Toeplitz operators on the polydisc,}
		{Bull. Lond. Math. Soc. 54 (2022), 1350--1362}.
		
		\bibitem{XD}{X. Ding and Y. Sang,}
		{\it Dual truncated Toeplitz operators,}
		{J. Math. Anal. Appl. {\bf{461}} (2018), no. 1, 929--946.}
		
		\bibitem{RD}{R. G. Douglas,} 
		{\it Banach algebra techniques in operator theory,} 
		{Second edition. Graduate Texts in Mathematics, 179. Springer-Verlag, New York, 1998}.
		
		\bibitem{GM}{S.R. Garcia, J. Mashreghi, and W. T. Ross,}
		{\it Introduction to model spaces and their operators,}
		{(Vol. 148), Cambridge University Press, 2016.}
		
		
		\bibitem{G-ROSS-W}{S. R. Garcia, W. T. Ross, and W. R. Wogen,}
		{Spatial isomorphisms of algebras of truncated Toeplitz operators,}
		{Indiana Univ. Math. J. $\bm{59}$ (2010), 1971--2000.}
		
		\bibitem{C Gu}{C. Gu,}
		{\it Characterizations of dual truncated Toeplitz operators,}
		{J. Math.\ Anal.\ Appl.\ $\bm{496}$ (2021),  no.\ 2, Paper No.\ 124815, 24 pp.}
		
		\bibitem{KLN}{E. Ko, J. E. Lee, and T. Nakazi}, 
		{\it On the dilation of truncated Toeplitz operators II}, 
		{Complex Anal. Op. Th., 13 (2019), 3549-3568.}
	
		\bibitem{LI-SANG-DING}{Y.~Li, Y. Q. Sang, and X. H. Ding,}
		{\it The commutant and invariant subspaces for dual truncated Toeplitz operators,}
		{Banach J. Math.\ Anal.\ $\bm{15}$ (2021), no.\ 1, Paper No.\ 17, 26 pp.}
		
		\bibitem{PM}{P. Ma, F. Yan, and D. Zheng,}
		{\it Zero, Finite rank, and compact big truncated Hankel operators on model spaces, Proc.\ Amer.\ Math.\ Soc. (12) $\bf{146}$ (2018), 5235--5242.}	
		
		\bibitem{NIK}{N. K. Nikolski,}
		{\it Treatise on the shift operator: spectral function theory,}
		{Springer, New York, 1986.}
		
		
		\bibitem{SANG-QIN-DING}{Y.~Sang, Y.~Qin, and X.~Ding,}
		{\it A theorem of Brown-Halmos type for dual truncated Toeplitz operators,}
		{Ann.\ Funct.\ Anal.\ (2019) Art.\ No.\ 2.}
		
		\bibitem{DS}{D. E. Sarason,} 
		{\it Algebraic properties of truncated Toeplitz operators,} 
		{Oper.\ Matrices $\bm{1}$ (2007), no. 4, 491--526.}
		
		
		\bibitem{NS}{N. A. Sedlock,} 
		{\it Algebras of truncated Toeplitz operators,} 
		{Oper.\ Matrices {\bf 5} (2011), no. 2, 309--326.}
		
		\bibitem{S-T-Z}{E. Strouse, D. Timotin, and M. Zarrabi,}
		{\it Unitary equivalence to truncated Toeplitz operators,}
		{Indiana Univ. Math. J. $\bm{61}$ (2012), 525--538.}
		
		\bibitem{XY}{X. Yang,}
		{\it The relationship between model spaces and the invariant subspaces of the unilateral shift,}
		{Contemp. Math. $\bm{5}$ (2024), 1474--1486.}
		
		
		
		
		
\end{thebibliography}
\end{document}